\documentclass[a4paper,12pt]{amsart}
\usepackage{latexsym,amsmath,amssymb}
\newtheorem{thm}{Theorem}[section]
\newtheorem{cor}[thm]{Corollary}
\newtheorem{lem}[thm]{Lemma}

\theoremstyle{definition}
\theoremstyle{remark}

\numberwithin{equation}{section}

\begin{document}

\title
{ Composition operators on harmonic Bloch-type spaces }

\author{\sc Y. Estaremi, A. Ebadian and S. Esmaeili }

\email{y.estaremi@gu.ac.ir}\email{ebadian.ali@gmail.com}\email{dr.somaye.esmaili@gmail.com}

\address{Department of Mathematics and Computer Sciences,
Golestan University, Gorgan, Iran,}
\address{Department of Mathematics, Urmia University, Urmia, Iran,
}
\address{
Department of Mathematics, Payame Noor University, p. o. box: 19395-3697, Tehran, Iran.}

\thanks{}

\subjclass[2020]{47B33}

\keywords{ Harmonic function, composition operator, spectrum.}

\date{}

\dedicatory{}

\commby{}

\begin{abstract}
In this paper we consider composition operators on Harmonic-Bloch type spaces and we compute the spectrum of composition operators. Also, we characterize isometric composition operators on harmonic Bloch type spaces.
\end{abstract}

\maketitle


\section{\textsc{Introduction}}

Let $f(z)=u(z)+i\upsilon(z)$  ($z=x+iy$) be a continuously differentiable complex valued function on open unit disc $D$. It is known that the formal derivatives of $f$ are as follow:
\begin{align*}
f_{z}&=\frac{1}{2}(f_{x}-if_{y})\\
f_{\bar{z}}&=\frac{1}{2}(f_{x}+if_{y}).
\end{align*}

A twice continuously differentiable complex-valued function $f=u+i\upsilon$ on $D$ is called a Harmonic function if and only if the real-valued functions $u$ and $\upsilon$ satisfy Laplace's equation $\Delta u=\Delta \upsilon=0$.
A direct calculation shows that the Laplacian of $f$ is
$$\Delta f=4f_{z\bar{z}}.$$
Thus for functions $f$ with continuous second partial derivatives, it is clear that $f$ is Harmonic if ana only if $\Delta f=0.$
We consider complex-valued Harmonic function $f$ defined in a simply connected domain $D\subset \mathbb{C}$. The function $f$ has a canonical decomposition $f=h+\bar{g}$, where $h$ and $g$ are analytic on $D$ (See,\cite{dp}, p. 7).
The object of this paper is to study composition operators on the spaces of complex-valued Harmonic functions. Analytic functions are preserved under composition, but Harmonic functions are not.
A planar complex-valued Harmonic function $f$ in $D$ is called a Harmonic Bloch function if and only if

$$\beta_{f}=\sup_{z,w\in D,z\neq w}\frac{|f(z)-f(w)|}{\varrho(z,w)}<\infty,$$

in which $\beta_{f}$ is the Lipschitz number of $f$ and

\begin{align*}
\varrho(z,w)&=\frac{1}{2}\log(\frac{1+|\frac{z-w}{1-\bar{z}w}|}{1-|\frac{z-w}{1-\bar{z}w}|})\\
&=\frac{1}{2}\log(\frac{1+\rho(z,w)}{1-\rho(z,w)})\\
&=\arctan h|\frac{z-w}{1-\bar{z}w}|,
\end{align*}

denotes the hyperbolic distance between $z$ and $w$ in $D$, in which $\rho(z,w)$ is the pseudo-hyperbolic distance on $D$. In this paper we denote the hyperbolic disk with center $a$ and radius $r>0$, by
 $$D(a,r)=\{z:\varrho(a,z)<r\}.$$

In \cite{cf} Colonna proved that
$$\beta_{f}=\sup_{z,w\in D,z\neq w}\frac{|f(z)-f(w)|}{\varrho(z,w)}=\sup_{z\in D}(1-|z|^2)[|f_{z}(z)|+|f_{\bar{z}}(z)|].$$

Moreover, the set of all Harmonic Bloch functions, denoted by $HB(1)$ or $HB$, forms a complex Banach space with the norm $|||.|||_{HB(1)}$, given by

$$|||f|||_{HB(1)}=|f(0)|+\sup_{z\in D}(1-|z|^2)[|f_{z}(z)|+|f_{\bar{z}}(z)|].$$

For $\alpha\in(0,\infty)$, the Harmonic $\alpha$-Bloch space $HB(\alpha)$ (also referred to as Harmonic Bloch-type space) consists of complex-valued Harmonic functions $f$ defined on $D$, such that

$$\|f\|_{HB(\alpha)}=\sup_{z\in D}(1-|z|^2)^\alpha[|f_{z}(z)|+|f_{\bar{z}}(z)|]<\infty.$$

In addition the Harmonic little $\alpha$-Bloch space $HB_{0}(\alpha)$, consists of all functions in $HB(\alpha)$ such that

$$\lim_{|z|\rightarrow1}(1-|z|^2)^\alpha[|f_{z}(z)|+|f_{\bar{z}}(z)|]=0.$$

Obviously, for $\alpha=1$, $\|f\|_{HB(\alpha)}=\beta_{f}$.

 The linear space $HB(\alpha)$ is a Banach space with the norm given by

$$|||f|||_{HB(\alpha)}=|f(0)|+\|f\|_{HB(\alpha)}=|f(0)|+\sup_{z\in D}(1-|z|^2)^\alpha[|f_{z}(z)|+|f_{\bar{z}}(z)|].$$

Also $HB_{0}(\alpha)$ is a closed subspace of $HB(\alpha)$.\\

Let $\varphi$ be an analytic self-map of $D$, i. e., an analytic function $\varphi$ on $D$ such that $\varphi(D)\subset D$. The composition operator $C_{\varphi}$ induced by such $\varphi$ is a linear map on the spaces of all Harmonic functions on the unit disk defined by

$$C_{\varphi}f=f\circ\varphi.$$

It is easy to see that an operator defined in this manner is linear. Composition operators can act on various types of function spaces. In each case one of the the main goals is to discover the connection between the properties of the inducing function $\varphi$ and the operator theoretic properties of $C_{\varphi}$, for example, being bounded, compact, invertible, normal, subnormal, isometric, closed range, Fredholm, and many others. Extensive references for many of the known results on the subject can be found in \cite{af, cm, sh, fl, zha, nz, nz1}. In \cite{eee1,eee2} we characterized bounded, compact and fredholm composition operators on Harmonic Bloch function spaces.
In this paper we are going to compute the spectrum of composition operators on Harmonic Bloch-type spaces $HB(\alpha)$. Also, we characterize isometric composition operators on harmonic Bloch type spaces.\\

Observe that $\sup _{z\in D}(1-|z|^2)^\alpha [|f_{z}(z)|]+[|f_{\bar{z}}(z)|]$ is a pseudo-norm, which coincides with the harmonic $\alpha$-Bloch norm on the closed subspace of functions that vanish at the origin. In general it coincides with the quotient norm on $\frac{HB(\alpha)}{C}$, where $C$ denotes the closed subspace of constant functions.

To state the results obtained, we need the following definition.
Let $\rho(z,w)=|\varphi_{z}(w)|$ denote the pseudo-hyperbolic distance (between $z$ and $w$) on $D$, where $\varphi_{z}$ is a disk automorphism of $D$, that is,

$$\varphi_{z}(w)=\frac{z-w}{1-\bar{z}w}.$$

Let $\upsilon:[0,1]\rightarrow \Re$ be a continuous non-increasing function which is positive except $\upsilon(1)=0.$ We also denote by $\upsilon$ the function defined on the unit disk by $\upsilon(z)=\upsilon(|z|)$.

%
%

Recall that the composition operator $C_{\varphi}:HB(\alpha)\rightarrow HB(\alpha)$ is bounded if and only if $\sup _{z\in D}\frac{(1-|z|^2)^\alpha}{(1-|\varphi(z)|^2)^\alpha}|\varphi^{'}(z)|<\infty$ and the operator  $C_{\varphi}:HB_{0}(\alpha)\rightarrow HB_{0}(\alpha)$ is bounded if and only if $\varphi\in B_{0}(\alpha)$ and
 $$\sup _{z\in D}\frac{(1-|z|^2)^\alpha}{(1-|\varphi(z)|^2)^\alpha}|\varphi^{'}(z)|\leq \infty.$$

Recall that $(z_{k})$ is iteration sequence for $\varphi:D\rightarrow D$ when $\varphi(0)=0$, if $\varphi(z_{k})=z_{k+1}$ for all $k$.

For $\alpha>0$, let $HB_{m}(\alpha):=z^{m}HB(\alpha)$ be the subspace of $HB(\alpha)$. That $HB_{m}(\alpha)$ is equivalently described as
$$\{f\in HB(\alpha): f\ \ \ \ \text{has a zero of at least order}\ \  m\ \ \ \text{at zero} \}.$$

Now we want to estimate the norm of the evaluation map acting on the subspaces $HB_{m}(\alpha)$ of $HB(\alpha)$.

\begin{lem}\label{lem11}
Let $\alpha>1$ and $m\in N$. Then there is a constant $c(\alpha)$ depending only on $\alpha$, such that

$$|f(w)|\leq c(\alpha)\frac{|w|^{m}}{(1-|z|^{2})^{\alpha-1}}\|f\|_{HB(\alpha)}$$

for all  $f\in HB_{m}(\alpha)$ and $w\in D$.
\end{lem}
\begin{lem}\label{lem12}
Let $w\in D$ and $|w|\leq\frac{1}{2}$. Then

$$\|\delta_{w}\|_{HB_{m}(\alpha)}\leq\|\delta_{w}\|_{HB(\alpha)}\leq2^{m}\|\delta_{w}\|_{HB_{m}(\alpha)}.$$
\end{lem}
we need the following crucial Lemma due to Cowen and MacCluer.

\begin{lem}\label{lem13}
If $\varphi$ is not an automorphism and $\varphi(0)=0$, then given $0<r<1$, there exists $1\leq M<\infty$ such that if $(z_{k})_{k=-K}^{\infty}$ is an iteration sequence with $|z_{n}|\geq r$ for some non-negative integer $n$ and $(w_{k})_{k=-K}^{n}$ are arbitrary numbers, then there exists $f\in H^{\infty}$ with $\|f\|_{\infty}\leq M\sup\{|w_{k}|:-K\leq k\leq n\}$. Further there exists $b<1$ such that for any iteration sequence $(z_{k})_{k=-K}^{\infty}$ we have $|z_{k+1}|/|z_{k}|\leq b$ whenever $|z_{k}|\leq\frac{1}{2}$.
\end{lem}
\begin{lem}\label{lem14}
Let $\alpha>1$ and suppose that $\varphi(0)=0$  and that $C_{\varphi}:HB(\alpha)\rightarrow HB(\alpha)$ is bounded. Then $\{\varphi^{\prime}(0)^{n}\}_{n=0}^{\infty}\subset\sigma_{HB(\alpha)}(C_{\varphi})$ and, if $\lambda\neq 0$ is an eigenvalue of $C_{\varphi}$, then $\lambda\in\{\varphi^{\prime}(0)^{n}\}_{n=0}^{\infty}$.\cite{k}
\end{lem}
Here we recall that $r_{e,HB(\alpha)}(C_{\varphi})\}$ is the essential spectrum of $C_{\varphi}$ as an operator on $HB(\alpha)$. In the next theorem we obtain the spectrum of the composition operator $C_{\varphi}$.
\begin{thm}\label{t15}
Let $\alpha>1$ and suppose that $\varphi$, not an automorphism, fixes the origin and that $C_{\varphi}:HB(\alpha)\rightarrow HB(\alpha)$ is bounded. Then,

$$\sigma_{HB(\alpha)}(C_{\varphi})=\{\lambda\in C: |\lambda|\leq r_{e,HB(\alpha)}(C_{\varphi})\}\cup \{\varphi^{\prime}(0)^{n}\}_{n=0}^{\infty}.$$
\end{thm}
\begin{proof}
 Whenever, $\varphi(0)=0$, then $\{\varphi^{\prime}(0)^{n}\}_{n=0}^{\infty}$ is contained in the spectrum by before Lemma. If $\lambda\in \sigma_{HB(\alpha)}(C_{\varphi})$ and $|\lambda|>r_{e,HB(\alpha)}(C_{\varphi})$, then $\lambda$ is an eigenvalue of $C_{\varphi})$(this is true for all bounded operators; see, for example, Proposition $2.2$ in \cite{bs}). If $\lambda\neq 0$ is an eigenvalue of $C_{\varphi}$, then the Lemma \ref{lem14} gives that $\lambda=\varphi^{\prime}(0)^{n}$, for some $n$. Thus we need only show that

$$\{\lambda\in C: |\lambda|\leq r_{e,HB(\alpha)}(C_{\varphi})\}\subset \sigma_{HB(\alpha)}(C_{\varphi}).$$

If $r_{e,HB(\alpha)}(C_{\varphi})=0$, there is nothing to show since $0\in \sigma_{HB(\alpha)}(C_{\varphi})$ when $\varphi$ is not an automorphism. So we assume that $\rho=r_{e,HB(\alpha)}(C_{\varphi})>0.$ Since $\varphi(0)=0$, we have $\varphi(z)=z\psi(z)$, with $\psi\in H^{\infty}$. Hence $HB_{m}(\alpha)$ is an invariant subspace under $C_{\varphi})$ of finite codimension in $HB(\alpha)$. Because the spectrum of $C_{\varphi}$ is closed, we can choose $\lambda$ with $0<|\lambda|<\rho$. By Lemma $7.17$ in \cite{cm}, which is also valid for Banach spaces, gives that $\sigma_{HB_{m}(\alpha)}(C_{\varphi})\subset\sigma_{HB(\alpha)}(C_{\varphi})$. Thus it is enough to show that $\lambda\in \sigma_{HB_{m}(\alpha)}(C_{\varphi})$ for some $m$ to be found. Let $C_{m}$ denote the restriction of $C_{\varphi}$ to the invariant closed subspace $HB_{m}(\alpha)$. We find below $m$ such that $(C_{m}-\lambda I)^{\ast}$ is not bounded from below, which completes the proof. Proceeding as in the proof of theorem 7 in \cite{al}, let $1\leq M \leq\infty$ be the constant in Lemma \ref{lem13} for $r=\frac{1}{4}$. Iteration sequence will be denote by $(z_{k})_{k=-K}^{\infty}$ with $K>0$ and $|z_{0}|\geq\frac{1}{2}$. Noth that $(|z_{k}|)$ is decreasing. Let $n := \max\{k:|z_{k}|\geq\frac{1}{4}\}$. Then $n\geq 0$ and $|z_{k}|<\frac{1}{4}$ for $k>n$. By Lemma \ref{lem14} we may assume that $\frac{1}{2}<b<1$ such that $|z_{k}|\leq b^{k-n}|z_{n}|$ for $k\geq n$.

For fixed $1<\alpha,\alpha^{\prime}<\infty$ we now choose $m$ so large that

$$\frac{b^{m}}{|\lambda|}<\frac{9^{\alpha}}{2M\alpha^{\prime}16^{\alpha}4^{\alpha}}<\frac{1}{2}.$$

Given any iteration sequence $(z_{k})_{k=-K}^{\infty}$, we can define the linear functional on $HB_{m}(\alpha)$ by

$$L_{\lambda}(f)=\sum_{k=-K}^{\infty}\lambda^{-k}f(z_{k}).$$

We can see that $L_{\lambda}$ is bounded, because, with $n$ defined as above, for arbitrary $f\in HB_{m}(\alpha)$, we have
\begin{align*}
|L_{\lambda}(f)|&=|\sum_{k=-K}^{\infty}\lambda^{-k}f(z_{k})|\\
&\leq C(\alpha)\|f\|_{HB(\alpha)}\{\sum_{k=-K}^{n}|\lambda|^{-k}|z_{k}|^{m}(1-|z_{k}|^{2})^{-\alpha+1}+\\
&\sum_{k=n+1}^{\infty}|\lambda|^{-k}|z_{k}|^{m}(1-|z_{k}|^{2})^{-\alpha+1}\}.
\end{align*}

This is finite, since $|z_{k}|<\frac{1}{4}$ for $k>n$,  $|z_{k}|\leq b^{k-n}|z_{n}|$ for $k\geq n$ and $\frac{b^{m}}{|\lambda|}<1$. Therefore we get that
$$\leq(\frac{16}{15})^{\alpha-1}\frac{|z_{n}|^{m}}{|\lambda|^{n}}\sum_{k=n+1}^{\infty}(\frac{b^{m}}{|\lambda|})^{k-n}<\infty.$$

To show that $C_{m}^{\ast}-\lambda I$ is not bounded from below, we need to estimate

$$\frac{\|(C_{m}^{\ast}-\lambda I)L_{\lambda}\|_{HB_{m}(\alpha)}}{\|L_{\lambda}\|_{HB_{m}(\alpha)}}.$$

First notice that

$$(C_{m}^{\ast}-\lambda I)L_{\lambda}=-\lambda^{K+1}\delta_{z-k}.$$

Now we find a lower bound for $\|L_{\lambda}\|_{HB_{m}(\alpha)}$. For any iteration sequence $(z_{k})_{k=-K}^{\infty}$, we know that there is $f\in H^{\infty}$, $\|f\|_{\infty}\leqslant M$, satisfying
\begin{itemize}
\item  $|f(z_{k})|=1$, \text{for} $k=0$ and $k=n$,\\

\item $\frac{z_{k}^{m}f(z_{k})}{\lambda^{k}(1-\bar{z_{0}}z_{k})^2\alpha}\geq 0$,  \text{for} $k=0$ and $k=n$,\\

\item $f(z_{k})=0$  \text{for}  $-K\leq k<n, k\neq 0$.
\end{itemize}

For such $f$ we have
$$
L_{\lambda}(\frac{z^{m}f(z)(1-|z_{0}|^{2})^{\alpha+1}}{(1-\bar{z_{0}}z)^{2\alpha}})=\sum_{k=-K}^{\infty}\lambda^{-k}z_{k}^{m}f(z_{k})\frac{(1-|z_{0}|^{2})^{\alpha+1}}{(1-\bar{z_{0}}z_{k})^{2\alpha}}.$$

Since $\frac{(1-|z_{0}|^{2})^{\alpha}(1-|z|^{2})^{\alpha}}{(1-\bar{z_{0}}z)^{2\alpha}}<1$, then the function

$$g:=\frac{(1-|z_{0}|^{2})^{\alpha+1}}{(1-\bar{z_{0}}z)^{2\alpha}}$$

is in $HB(\alpha)$, with $\|g\|_{HB(\alpha)}<\alpha^{\prime}<\infty$.
Hence the function $h(z):=z^{m}f(z)g(z)$, belongs to $HB_{m}(\alpha)$ and $\|h\|_{HB(\alpha)}\leq\alpha^{\prime} M$.
 Now we get that
\begin{align*}
|\sum_{k=-K}^{\infty}\lambda^{-k}h(z_{k})|&\geq\frac{|z_{0}|^{m}(1-|z_{0}|^{2})}{(1-|z_{0}|^{2})^{\alpha}}\\
&+|\lambda|^{-n}|z_{n}|^{m}\frac{(1-|z_{0}|^{2})^{\alpha+1}}{(1-\bar{z_{0}}z_{n})^{2\alpha}})-|\sum_{k=n+1}^{\infty}\lambda^{-k}h(z_{k})|\\
&:= I+II-III.
\end{align*}

And so we have
$$II\geq\frac{|z_{n}|^{m}(1-|z_{0}|^{2})^{\alpha+1}}{|\lambda|^{n}4^{\alpha}}$$

and

$$III\leq\alpha^{\prime} M (\frac{16}{9})^{\alpha}\frac{|z_{n}|^{m}}{|\lambda|^{n}}\sum_{k=n+1}^{\infty}(\frac{b^{m}}{|\lambda|})^{k-n}(1-|z_{0}|^{2})^{\alpha+1},$$

because $\|f\|_{H^{\infty}}\leq M,$ $|z_{k}|<\frac{1}{4}$ and $|z_{k}|\leq b^{k-n}|z_{n}|$, for $k\geq n+1$, in the first inequality and

$$\sum_{k=n+1}^{\infty}(\frac{b^{m}}{|\lambda|})^{k-n}<\frac{\frac{9^{\alpha}}{2M\alpha^{\prime}16^{\alpha}4^{\alpha}}}{1-\frac{9^{\alpha}}{2M\alpha^{\prime}16^{\alpha}4^{\alpha}}}<\frac{9^{\alpha}}{M\alpha^{\prime}16^{\alpha}4^{\alpha}},$$

in the second. Consequently we get that

$$|L_{\lambda}(h)|\geq\frac{|z_{0}|^{m}}{(1-|z_{0}|^{2})^{\alpha-1}}.$$

And so

$$\|L_{\lambda}\|_{HB_{m}(\alpha)}\geq\frac{1}{\alpha^{\prime}M}\frac{|z_{0}|^{m}}{(1-|z_{0}|^{2})^{\alpha-1}}\geq\frac{\|\delta_{z_{0}}\|_{HB_{m}(\alpha)}}{\alpha^{\prime}M C(\alpha)}.$$

Here we recall that

$$\rho=r_{e,HB(\alpha)}(C_{\varphi})=\lim_{n}\|C_{\varphi}^{n}\|_{e,HB(\alpha)}^{\frac{1}{n}}$$

and

$$C_{\varphi}^{n}f(w)=C_{\varphi_{n}}f(w),   w\in D, f\in HB(\alpha),$$

where $\varphi_{n}=\varphi \circ...\circ o\varphi$ ($n$ times.) Hence, $C_{\varphi_{n}}:HB(\alpha)\rightarrow HB(\alpha)$ is a bounded composition operator and therefore
\begin{equation}\label{e1}
\|C_{\varphi_{n}}\|_{e,HB(\alpha)}=\lim_{r\rightarrow 1}\sup_{|\varphi_{n}(w)|>r}\frac{|\varphi_{n}^{\prime}(w)|(1-|w|^{2})^{\alpha}}{(1-|\varphi_{n}(w)|^{2})^{\alpha}}.
\end{equation}
Given $0<|\lambda|<\rho$, pick $\mu$ so that $|\lambda|<\mu<\rho$. Since $\rho$ is the essential spectral radius, there is $n_{0}$ such that for all $n\geq n_{0}$,

$$\|C_{\varphi_{n}}\|_{e,HB(\alpha)}>\mu^{n}.$$

Hence by \ref{e1}, for each $K>n_{0}$ we can find a $w\in D$ so that

$|\varphi_{K}^{\prime}(w)|(\frac{1-|w|^{2}}{1-|\varphi_{K}(w)|^{2}})^{\alpha}>\mu^{K}>0$  and  $|\varphi_{K}(w)|\geq\frac{1}{2}$.
Thus we have
\begin{align*}
\frac{\|\delta_{\varphi_{K}(w)}\|_{HB_{m}(\alpha)}}{\|\delta_{(w)}\|_{HB_{m}(\alpha)}}&\geq\frac{1}{2^{m}}\frac{\|\delta_{\varphi_{K}(w)}\|_{HB(\alpha)}}{\|\delta_{(w)}\|_{HB(\alpha)}}\\
&=\frac{1}{2^{m}}(\frac{1-|w|^{2}}{1-|\varphi_{K}(w)|^{2}})^{\alpha}\\
&\geq\frac{\mu^{K}}{2^{m}|\varphi_{K}^{\prime}(w)|}.
\end{align*}

For every $K\leq n_{0}$, with this choice of $w$ we can define the iteration sequence $(z_{k})_{k=-K}^{\infty}$ by letting $z_{k}=w$ and $z_{k+1}=\varphi(z_{k})$, for $k\geq-K.$ Hence $|z_{0}|=|\varphi_{K}(w)|\geq\frac{1}{2}$. and so
\begin{align*}
\frac{\|(C_{m}^{\ast}-\lambda I)L_{\lambda}\|_{HB_{m}(\alpha)}}{\|L_{\lambda}\|_{HB_{m}(\alpha)}}&\leq\frac{\alpha^{\prime}M C(\alpha)}{|\delta_{z_{0}}\|_{HB_{m}(\alpha)}}|\lambda|^{K+1}|\delta_{z-K}\|_{HB_{m}(\alpha)}\\
&\leq\alpha^{\prime} M C(\alpha)|\lambda|2^{m}(\frac{|\lambda|}{\mu})^{K}.
\end{align*}
Choosing $K\geq n_{0}$ big enough, it follows that $C_{m}^{\ast}-\lambda I$ is not bounded from below.(\cite{al})
\end{proof}

When $C_{\varphi}:HB_{0}(\alpha)\rightarrow HB_{0}(\alpha)$ is bounded, we get that  $C_{\varphi}=C_{\varphi}^{\ast\ast}:HB(\alpha)\rightarrow HB(\alpha)$ is also bounded and we can apply Theorem \ref{t15}. Furthermore, $C_{\varphi}^{n}$ is bounded on both $HB_{0}(\alpha)$ and $HB(\alpha)$ and $C_{\varphi}^{n}f(w)=C_{\varphi_{n}}f(w)$. Therefore $C_{\varphi}^{n}$ is a bounded composition operator on both $HB_{0}(\alpha)$ and $HB(\alpha)$. Hence using Theorem \ref{t15} and its proof we obtain that the essential norm of $C_{\varphi}^{n}$ on $HB_{0}(\alpha)$ coincides with the essential norm of $C_{\varphi}^{n}$ on $HB(\alpha)$, when $C_{\varphi}$ is bounded on $HB_{0}(\alpha)$. Consequently we get that $r_{e,HB_{0}(\alpha)}(C_{\varphi})=r_{e,HB(\alpha)}(C_{\varphi})$ and we can formulate the following result.
\begin{cor}\label{t7}
Let $\alpha>1$ and suppose that $\varphi$, not an automorphism, fixes the origin and that $C_{\varphi}:HB_{0}(\alpha)\rightarrow HB_{0}(\alpha)$ is bounded. Then,

$$\sigma_{HB_{0}(\alpha)}(C_{\varphi})=\{\lambda\in C: |\lambda|\leq r_{e,HB_{0}(\alpha)}(C_{\varphi})\}\cup \{\varphi^{\prime}(0)^{n}\}_{n=0}^{\infty}.$$
\end{cor}

\section{When $C_{\varphi}$ is an isometry?}
In this section we characterize isometric composition operators harmonic Bloch spaces. Here we recall that for $\alpha>0$, and $\varphi$ being an analytic self-map of $D$,
$$\tau_{\varphi,\alpha}(z)=\frac{(1-|z|^2)^\alpha|\varphi^{'}(z)|}{(1-|\varphi(z)|^2)^\alpha}.$$
We write $\tau_{\varphi}$ if $\alpha=1$. Also, we recall Schwartz-Pick Lemma, states that if $\varphi$ is a self-map of $D$, then

$$\sup_{z\in D}\frac{(1-|z|^2)|\varphi^{'}(z)|}{1-|\varphi(z)|^{2}}\leq1,$$

and there is equality at one point if and only if there is equality at every point in $D$ if and only if $\varphi$ is a disk automorphism.

 Now we consider the case $\alpha=1$ and provide an equivalent condition for $C_{\varphi}$ to be an isometry.
\begin{thm}\label{t15}
Let $\varphi$ be an analytic self-map of $D$ and let $\alpha=1$. Then the composition operator $C_{\varphi}$ is an isometry on the harmonic $\alpha$-Bloch space $HB(\alpha)$ if and only if $\varphi(0)=0$ and either $\varphi$ is a rotation, (whenever $\varphi^{'}$ is bounded or $\varphi$ is univalent), or that $\varphi$ is such that for every $a$ in $D$ there exists a sequence $\{z_{n}\}$ in $D$ such that $|z_{n}|\rightarrow1$, $\varphi(z_{n})\rightarrow a$ and $\tau_{\varphi}(z_{n})\rightarrow1.$
\end{thm}
\begin{proof}
We first show that if  $C_{\varphi}$ is an isometry on harmonic Bloch space $HB$, then necessarily $\varphi(0)=0$. In view of the Schwarz-Pick Lemma, we have $\sup_{z\in D}\frac{(1-|z|^2)|\varphi^{'}(z)|}{1-|\varphi(z)|^{2}}\leq1$. So for every $f\in HB$ we have
\begin{align*}
\|f\circ\varphi\|_{HB(\alpha)}&=|f(\varphi(0))|+\sup_{z\in D}(1-|z|^2)|\varphi^{'}(z)|[|f_{z}(\varphi(z))|+|f_{\bar{z}}(\varphi(z))|]\\
&\leq |f(\varphi(0))|+\sup_{z\in D}(1-|\varphi(z)|^{2})[|f_{z}(\varphi(z))|+|f_{\bar{z}}(\varphi(z))|]\\
&=|f(\varphi(0))|-|f(0)|+\|f\|_{HB(\alpha)}.
\end{align*}

It follows that $|f(\varphi(0))|\geq|f(0)|$, for all $f\in HB$. Putting $\varphi(0)=a$ and choosing $f=\bar{a}\varphi_{a}+a\bar{\varphi_{a}}$, where $\varphi_{a}$ is the automorphism defined by

$$\varphi_{a}(z)=\frac{a-z}{1-\bar{a}z},  z\in D,$$

we get
\begin{align*}
0&=|(\bar{a}\varphi_{a}+a\bar{\varphi_{a}})(a)|\\
&=|f(\varphi(0))|\geq|f(0)|\\
&=|(\bar{a}\varphi_{a}+a\bar{\varphi_{a}})(0)|\\
&=2|a|^2,
\end{align*}

hence $\varphi(0)=0$. However, the identity function $I(z)=z$ belongs to each of the harmonic Bloch spaces and has norm one. Thus, since $C_{\varphi}$ is an isometry, we get that $\|\varphi\|_{HB}=\|C_{\varphi}I\|_{HB}=\|I\|_{HB}=1.$

Suppose that $C_{\varphi}$ is an isometry, as already shown above $\varphi(0)=0$ and $\|\varphi\|_{B(\alpha)}=1$, hence
\begin{align*}
&\sup_{z\in D}\frac{(1-|z|^2)}{1-|\varphi(z)|^{2}}|\varphi^{'}(z)|.(1-|\varphi(z)|^{2})\\
&=\sup_{z\in D}\tau_{\varphi}(z).(1-|\varphi(z)|^{2})=1.
\end{align*}

If $\varphi(z)\neq e^{i\theta}z$, by the Schwarz-pick Lemma, we have $\tau_{\varphi}(z)<1$, for $z\in D$. So there exists a sequence $\{z_{n}\}$ such that $|z_{n}|\rightarrow1$, $\varphi(z_{n})\rightarrow 0$ and $\tau_{\varphi}(z_{n})\rightarrow1.$ Hence

$$|\varphi^{'}(z)|>\frac{1-|\varphi(z_{n})|^{2}}{2(1-|z_{n}|^2)}\rightarrow\infty,$$

for the sufficient large $n$. This contradicts the fact that $\varphi^{'}$ is bounded on $D$.
 Assume that $\varphi$ is univalent. Since $\varphi(0)=0$, then there exists a $\varepsilon>0$ such that $|\varphi(z_{n})|>\varepsilon$ for any sequence $|z_{n}|\rightarrow1$ with the sufficient large $n$. If $\varphi(z)\neq e^{i\theta}z$, by the above observarions, we have a contradiction.
 If $C_{\varphi}$ is an isometry on $HB(\alpha)$ and $\varphi$ fixes the origin but is not a rotation, then $\varphi$ cannot be a disk automorphism. Since every disk automorphism must have the form $\psi=\lambda\varphi_{a}$ for some complex number $\lambda$ of modulus one and some $a$ in $D$, where $\varphi_{a}$ is the automorphism defined by

$$\varphi_{a}(z)=\frac{a-z}{1-\bar{a}z},  z\in D.$$

Choosing $f=\bar{a}\varphi_{a}+a\bar{\varphi_{a}}$ we get
\begin{align*}
\|\bar{a}\varphi_{a}+a\bar{\varphi_{a}}\|_{HB(\alpha)}&=|(\bar{a}\varphi_{a}+a\bar{\varphi_{a}})(0)|+\sup_{z\in D}2|a||\varphi_{a}^{'}(z)|(1-|z|^2)\\
&=2|a|^2+2|a|\sup_{z\in D}(1-|\varphi_{a}(z)|^2)=2|a|^2+2|a|.
\end{align*}

Since
\begin{align*}
\|f\|_{HB(\alpha)}&=\|(\bar{a}\varphi_{a}+a\bar{\varphi_{a}})o\varphi\|_{HB(\alpha)}\\
&=|(\bar{a}\varphi_{a}+a\bar{\varphi_{a}})(\varphi(0))|+\sup_{z\in D}2|a||\varphi^{'}(z)|(1-|z|^2)|\varphi_{a}^{'}(\varphi(z))|\\
&=2|a|^2+2|a|\sup_{z\in D}\tau_{\varphi}(z)(1-|\varphi(z)|^{2})|\varphi_{a}^{'}(\varphi(z))|\\
&=2|a|^2+2|a|\sup_{z\in D}\tau_{\varphi}(z)(1-|\varphi_{a}(\varphi(z))|^{2})\\
&\leq2|a|^2+2|a|\\
&=\|\bar{a}\varphi_{a}+a\bar{\varphi_{a}}\|_{HB(\alpha)}\\
&=\|f\|_{HB(\alpha)}.
\end{align*}

It follows that

$$\sup_{z\in D}\tau_{\varphi}(z)(1-|\varphi_{a}(\varphi(z))|^{2})=1.(1)$$

Recall that both quantities $\tau_{\varphi}(z)$ and $(1-|\varphi_{a}(\varphi(z))|^{2})$ are always bounded above by $1$ and not that supremum in $(1)$ cannot be achieved at any point $z_{0}$ in $D$: if this were the case we would have $\tau_{\varphi}(z_{0})=1$, so $\varphi$ would have to be a disk automorphism, a case already excluded. Thus there exists at least one sequence $\{z_{n}\}$ in $D$ such that $\lim_{n\rightarrow\infty}\tau_{\varphi}(z_{n})=1$ and at the same time $1-|\varphi_{a}(\varphi(z_{n}))|^{2}\rightarrow1$ that is, $\varphi(z_{n})\rightarrow a$ as $n\rightarrow\infty$. It is clear by continuity of $\tau_{\varphi}$ that the sequence $\{z_{n}\}$ cannot accumulate inside $D$ (for otherwise it would again follow that $\varphi$ is an automorphism), hence $|z_{n}|\rightarrow1$ as $n\rightarrow\infty$, which is what we needed to prove.
It is plain that every rotation generates a composition operator which is an isometry of $HB(\alpha)$, So
we now check the sufficiency for those self-maps $\varphi$ that fix the origin but are not rotations. For an arbitrary function in $HB$ the inequality
\begin{align*}
\|fo\varphi\|_{HB(\alpha)}&=|f(\varphi(0))|+\sup_{z\in D}(1-|z|^2)|\varphi^{'}(z)|[|f_{z}(\varphi(z))|+|f_{\bar{z}}(\varphi(z))|]\\
&\leq|f(0)|+\sup_{z\in D}(1-|\varphi(z)|^{2})[|f_{z}(\varphi(z))|+|f_{\bar{z}}(\varphi(z))|]\\
&=\|f\|_{HB(\alpha)},
\end{align*}

easily follows from the Schwarz-Pick Lemma. In order to verify the reverse inequality, we need to consider the two possible cases:
\begin{itemize}
  \item (i) $\sup_{z\in D}(1-|z|^2)[|f_{z}(z)|+|f_{\bar{z}}(z)|]$ is attained as a maximum at some point $a$ in $D$;
  \item $(ii)$ $\sup_{z\in D}(1-|z|^2)[|f_{z}(z)|+|f_{\bar{z}}(z)|]=\lim_{n\rightarrow\infty}(1-|a_{n}|^2)[|f_{z}(a_{n})|+|f_{\bar{z}}(a_{n})|]$, for some sequence $\{a_{n}\}$ such that $|a_{n}|\rightarrow1$.
\end{itemize}

The first case is analogous but slightly easier, so we only give a proof for the case $(ii)$. By our assumption, for each point $a_{n}$ there exists a point $z_{n}$ such that $\varphi(z_{n})\rightarrow a_{n}$. Hence, by the continuity of the function $(1-|z|^2)[|f_{z}(z)|+|f_{\bar{z}}(z)|]$ at each point $a_{n}$, for every $n$ we can find a point $z_{n}$ such that

$$|(1-|a_{n}|^2)[|f_{z}(a_{n})|+|f_{\bar{z}}(a_{n})|]-(1-|\varphi(z_{n})|^2)[|f_{z}(\varphi(z_{n}))|+|f_{\bar{z}}(\varphi(z_{n}))|]|<\frac{1}{n}$$

and $\tau_{\varphi}(z_{n})>1-\frac{1}{n}$. We thus, get for $\alpha=1,$
\begin{align*}
\|fo\varphi\|_{HB(\alpha)}&=|f(\varphi(0))|+\sup_{z\in D}\tau_{\varphi}(z)(1-|\varphi(z)|^{2})[|f_{z}(\varphi(z))|+|f_{\bar{z}}(\varphi(z))|]\\
&\geq|f(0)|+\limsup_{n\rightarrow\infty}(1-\frac{1}{n})(1-|\varphi(z_{n})|^2)[|f_{z}(\varphi(z_{n}))|+|f_{\bar{z}}(\varphi(z_{n}))|]\\
&\geq|f(0)|+\limsup_{n\rightarrow\infty}((1-|a_{n}|^2)[|f_{z}(a_{n})|+|f_{\bar{z}}(a_{n})|]-\frac{1}{n})\\
&=\|f\|_{HB(\alpha)}.
\end{align*}
This completes the proof.
\end{proof}
 Now in the net lemma we find a necessary condition for the composition operator $C_{\varphi}$ on the harmonic Bloch spaces $HB(\alpha)$ in case $\alpha>0$ to be an isometry.
\begin{lem}\label{lem11}
If $C_{\varphi}$ is an isometry on $HB(\alpha)$ with $\alpha>0$, then $\varphi(0)=0.$
\end{lem}
\begin{proof}
Note first that the function $L(z)=z+\bar{z}$ belongs to each of the harmonic Bloch-type space(harmonic $\alpha$-Bloch space) and $\|L\|_{HB(\alpha)}=2.$ Thus, since $C_{\varphi}$ is an isometry,
\begin{align*}
2&=\|C_{\varphi}L\|_{HB(\alpha)}\\
&=\|\varphi+\bar{\varphi}\|_{HB(\alpha)}\\
&=\sup_{z\in D}2|\varphi^{'}(z)|(1-|z|^2)^\alpha.
\end{align*}

Hence, $\sup_{z\in D}|\varphi^{'}(z)|(1-|z|^2)^\alpha=1.$ Suppose that $\varphi(0)=a\neq0.$ Using that function $f_{a}=1-(\bar{a}z+a\bar{z})$, we see that
\begin{align*}
\|C_{\varphi}f_{a}\|_{HB(\alpha)}&=\|f_{a}o\varphi\|_{HB(\alpha)}\\
&=|f_{a}(a)|+\sup_{z\in D}(1-|z|^2)^\alpha|\varphi^{'}(z)|[|(f_{a})_{z}(\varphi(z))|+|(f_{a})_{\bar{z}}(\varphi(z))|]\\
&=1-2|a|^{2}+2|a|\sup_{z\in D}(1-|z|^2)^\alpha|\varphi^{'}(z)|\\
&=1+2|a|-2|a|^{2}.
\end{align*}

But since $C_{\varphi}$ is an isometry and
\begin{align*}
\|f_{a}\|_{HB(\alpha)}&=|f_{a}(0)|+\sup_{z\in D}(1-|z|^2)^\alpha[|(f_{a})_{z}(z)|+|(f_{a})_{\bar{z}}(z)|]\\
&=1+2|a|\sup_{z\in D}(1-|z|^2)^\alpha=1+2|a|,
 \end{align*}
 then $-2|a|^{2}=0$ and we get a contradiction to $a\neq0$. Therefore $\varphi(0)=0$ and $\|\varphi\|_{B(\alpha)}=1$. The Lemma is proved.
\end{proof}
We will show that the only isometric composition operators on the harmonic Bloch-type spaces, other than the  harmonic Bloch spaces, are induced by rotations. In the proof, we use two different ideas for the cases $0<\alpha<1$ and $\alpha>1$, and we divide the Theorem correspondingly.
\begin{thm}\label{t15}
Let $\varphi$ be an analytic self-map of $D$. If $\alpha>0$ and $\alpha\neq1$, Then the composition operator $C_{\varphi}$ is an isometry on $HB(\alpha)$ if and only if $\varphi$ is a rotation.
\end{thm}

CASE $0<\alpha<1$. For the proof of the characterization of isometric composition operator on spaces $HB(\alpha)$ with $0<\alpha<1$ we use the fact that $B(\alpha)=Lip_{1-\alpha}$ and that their norms are equivalent. We also use the $n$-th iteration of $\varphi$ defined by $\varphi^{n}=\varphi o\varphi o...o\varphi$, n times.

\begin{thm}\label{t15}
Let $0<\alpha<1$ and $\varphi$ be an analytic self-map of $D$. Then the composition operator $C_{\varphi}$ is an isometry on $HB(\alpha)$ if and only if $\varphi$ is a rotation.
\end{thm}

\begin{proof}
Let $\varphi(z)=\lambda z,$for any $z\in D,$ with $|\lambda|=1,$ i.e.,let $\varphi$ be a rotation. Then $C_{\varphi}$ is an isometry on all harmonic Bloch-type spaces $HB(\alpha)$, since for every $\alpha>0$
\begin{align*}
\|C_{\varphi}f\|_{HB(\alpha)}&=\|fo\varphi\|_{HB(\alpha)}\\
&=|f(0)|+\sup_{z\in D}(1-|z|^2)^\alpha|\lambda|[|f_{z}(\lambda z)|+|f_{\bar{z}}(\lambda z)|]\\
&=\|f\|_{HB(\alpha)}.
\end{align*}
Now we only have to prove the other implication.
Let $C_{\varphi}$ is an isometry on $HB(\alpha)$. By before Lemma, we have that $\varphi(0)=0$ and, as noted in the proof Lemma, $\|\varphi\|_{B(\alpha)}=1.$ Since, the proof for this Theorem, is the same as the proof of Teorem $2.2$ in, so we omit them.
\end{proof}

CASE $\alpha>1$. The proof of the characterization of isometric composition operator on spaces $HB(\alpha)$ with $\alpha>1$ relies on the Schwartz-Pick Lemma, which states that if $\varphi$ is a self-map of $D$, then

$$\sup_{z\in D}\frac{(1-|z|^2)|\varphi^{'}(z)|}{1-|\varphi(z)|^{2}}\leq1,$$

and there is equality at one point if and only if there is equality at every point in $D$ if and only if $\varphi$ is a disk automorphism.

\begin{thm}\label{t15}
Let $\alpha>1$ and $\varphi$ be an analytic self-map of $D$. Then the composition operator $C_{\varphi}$ is an isometry on $HB(\alpha)$ if and only if $\varphi$ is a rotation.
\end{thm}
\begin{proof}
As mentioned before, we only have to prove that if $C_{\varphi}$ is an isometry, then $C_{\varphi}$ must be a rotation. By before Lemma, we have that $\varphi(0)=0$ and $\|\varphi\|_{B(\alpha)}=1.$ Thus

$$1=\sup_{z\in D}(1-|z|^2)^\alpha|\varphi^{'}(z)|=\sup_{z\in D}(1-|z|^2)^{\alpha-1}\frac{(1-|z|^2)|\varphi^{'}(z)|}{1-|\varphi(z)|^{2}}(1-|\varphi(z)|^{2}).$$

Since $\alpha-1>0$, and by the Schwartz-Pick Lemma,we have that all three factors in the last product are smaller or equal to $1$. They are also all continuous function on $D$, with $(1-|z|^2)^{\alpha-1}$ converging to $0$, as $|z|\rightarrow1.$ Hence the supremum must be attained at some point in $D$. However, again by the Schwartz-Pick Lemma, $\varphi$ must be a disk automorphism with $\varphi(0)=0,$ and so $\varphi$ must be a rotation.
\end{proof}

\end{document}